\documentclass[12pt]{article}
\usepackage{tikz}
\usepackage{jeffe}
\usepackage{verbatim}
\usepackage{amsmath}

\date{July 20, 2009}
\begin{document}

\def\root{1.7320508075}
\def\dot{circle (.10cm)}
\def\code{{\mathcal D}}
\def\part{{\mathcal P}}

\newcommand{\grid}[3]{
\begin{scope}[yshift=5cm]
\foreach \row in {1, 4,..., #1}
{
	\foreach \x in {1.73205, 3.4641, ..., #2}
	{
	\begin{scope}[] 
		\draw (\x, \row) +(30:1cm) -- +(90:1cm) -- +(150:1cm) -- +(210:1cm) -- +(270:1cm) -- +(330:1cm) -- cycle;
	\end{scope}
	}
	\foreach \x in {1.73205, 3.4641, ..., #3}
	{
	\begin{scope}[xshift = 0.8660254 cm, yshift = -1.5 cm]
		\draw (\x, \row) +(30:1cm) -- +(90:1cm) -- +(150:1cm) -- +(210:1cm) -- +(270:1cm) -- +(330:1cm) -- cycle;
	\end{scope}
	}
}
	\foreach \x in {1.73205, 3.4641, ..., #2}
	{
	\begin{scope}[yshift = -2cm]
		\draw (\x, 0) +(30:1cm) -- +(90:1cm) -- +(150:1cm) -- +(210:1cm) -- +(270:1cm) -- +(330:1cm) -- cycle;
	\end{scope}
	}
\end{scope}
}

\newcommand{\gridsmall}[2]{
\begin{scope}[yshift=5cm]
	\foreach \x in {1.73205, 3.4641, ..., #1}
	{
	\begin{scope}[] 
		\draw (\x, 1) +(30:1cm) -- +(90:1cm) -- +(150:1cm) -- +(210:1cm) -- +(270:1cm) -- +(330:1cm) -- cycle;
	\end{scope}
	}
	\foreach \x in {1.73205, 3.4641, ..., #2}
	{
	\begin{scope}[xshift = 0.8660254 cm, yshift = -1.5 cm]
		\draw (\x, 1) +(30:1cm) -- +(90:1cm) -- +(150:1cm) -- +(210:1cm) -- +(270:1cm) -- +(330:1cm) -- cycle;
	\end{scope}
	}
	\foreach \x in {1.73205, 3.4641, ..., #1}
	{
	\begin{scope}[yshift = -2cm]
		\draw (\x, 0) +(30:1cm) -- +(90:1cm) -- +(150:1cm) -- +(210:1cm) -- +(270:1cm) -- +(330:1cm) -- cycle;
	\end{scope}
	}
\end{scope}
}

\title{A New Lower Bound on the Density of \\Vertex Identifying Codes for the Infinite Hexagonal Grid}
\author{Daniel W. Cranston\thanks{
Department of Mathematics \& Applied Mathematics, Virginia Commonwealth University, Richmond, VA, 23284;  
Center for Discrete Mathematics and Theoretical Computer Science, Rutgers, Piscataway, NJ 08854. 
Email: \texttt{dcranston@vcu.edu}} 
\and Gexin Yu\thanks{
Department of Mathematics, College of William and Mary, Williamsburg, VA 23185.
Email: \texttt{gyu@wm.edu}. Research supported in part by NSF grant DMS-0852452.
}}
\maketitle

\begin{abstract}
Given a graph $G$, an identifying code $\code\subseteq V(G)$ is a vertex set such that for any two distinct vertices $v_1,v_2\in V(G)$, the sets $N[v_1]\cap\code$ and $N[v_2]\cap\code$ are distinct and nonempty (here $N[v]$ denotes a vertex $v$ and its neighbors).  We study the case when $G$ is the infinite hexagonal grid $H$.  Cohen et.al. constructed two identifying codes for $H$ with density $3/7$ and proved that any identifying code for $H$ must have density at least $16/39\approx0.410256$. Both their upper and lower bounds were best known until now.  Here we prove a lower bound of $12/29\approx0.413793$.
\end{abstract}
\bigskip

\newtheorem{theorem}{Theorem}
\newtheorem{claim}{Claim}
\newtheorem{lemma}{Lemma}
\newtheorem{prop}{Proposition}

\section{Introduction}
Identifying codes were introduced by Karpovsky et al.~\cite{KCL98} in 1998 to model fault diagnosis in multiprocessor systems.   If we model a multiprocessor as an undirected simple graph $G$,  then an {\em $(r,\le\!\ell)$-ID code} is a subset of the vertices of $G$ having the property that every  collection of at most $\ell$ vertices has a non-empty and distinct set of code vertices that are distance at most $r$ from it.    To be precise, let $N_r[X]$ be the set of vertices that are within distance $r$ of $X$ (called the ``closed $r$-neighborhood'').  An $(r,\le\!\ell)$-ID code is a subset $\code$ of $V(G)$ such that $N_r[X]\cap \code$ and $N_r[Y]\cap \code$ are distinct and non-empty for all distinct subsets $X\subseteq V(G)$ and $Y\subseteq V(G)$ with $|X|,|Y|\le\ell$.
In this paper, we consider the case $r=1$, which we denote simply as {\em $\ell$-ID codes}; we also write $N[v]$ for $N_1[v]$.

Not every graph has an $\ell$-ID code.  For example, if $G$ contains two vertices $u$ and $v$ such that $N[u] = N[v]$, then $G$ cannot have even a 1-ID code, since for any subset of vertices $\code$ we have $N[u]\cap \code =N[v]\cap \code$.  However, if, for every pair of subsets $X\not=Y$ with $|X|, |Y|\le \ell$,  we have $N[X]\not=N[Y]$, then $G$ has an $\ell$-ID code, since $\code=V(G)$ is such a code.  Hence we are usually interested in finding an $\ell$-ID code of minimum cardinality.    The most studied case is when $\ell=1$.  In this case, $\code$ is an $1$-ID code if and only if for all distinct vertices $u$ and $v$ in G, the intersections $N[u]\cap \code$ and $N[v]\cap \code$ are distinct and nonempty.  For a fixed subset of vertices $\code$, we say that vertices $u$ and $v$ are {\it distinguishable} if $N[u]\cap\code \ne N[v]\cap\code$.

Much work has focused on finding $1$-ID codes for infinite grids, see~\cite{BL05, CGHLMPZ00, CHLZ00, HL04a, HL04b}.   To measure how small an ID code can be,  we talk about the ``density'' of an infinite grid, which, roughly speaking, is the fraction of the vertices in the graph that are in the code (we give a formal definition of density after we prove Proposition~1).    

In 1998, Karpovsky, Chakrabarty, and Levitin~\cite{KCL98} considered the 6-regular, 4-regular, and 3-regular infinite grids that come from the tilings of the plane by equilateral triangles, squares, and regular hexagons. They asked the question ``What is the minimum density of an identifying code for each grid?''  A short proof shows that the answer for the 6-regular grid is density $1/4$. For the 4-regular grid,   Cohen et.~al~\cite{CGHLMPZ00} constructed codes with density $7/20$ and Ben-Haim and Litsyn~\cite{BL05} proved that $7/20$ is best possible.    For the 3-regular grid, the minimum density remains unknown.
The best upper bound is $3/7$, which comes from two codes constructed by Cohen et.~al~\cite{CHLZ00}; these same authors also proved a lower bound of $16/39$.  In this paper, we improve the lower bound to $12/29$.
Before we prove our main result, which is Theorem~\ref{mainthm}, we first prove a weaker lower bound, given in Proposition~\ref{prop1}.  The proof of Proposition~\ref{prop1} is instructive, because the proof is easy, yet the proofs of Theorem~\ref{mainthm} and Proposition~\ref{prop1} use the same core idea.
We call the components of $G[\code]$ {\it clusters} and a cluster with $d$ vertices is a {\it $d$-cluster} (a $d^+$-cluster has $d$ or more vertices).

\begin{prop}
\label{prop1}
The density of every vertex identifying code for the infinite hexagonal grid is at least 2/5.
\end{prop}

\begin{proof}
Our proof is by the discharging method.
Thus, $\code$ must contain at least $2/5$ of the vertices.
We assign to each vertex $v$ in the identifying code $\code$ a charge of 1.  We will redistribute the charges, without introducing any new charge, so that every vertex (whether in $\code$ or not) has charge at least 2/5.
We redistribute the charge according to the following discharging rule:

\begin{itemize}
\item
If $v\in\code$ is adjacent to $w\notin\code$ and $w$ has $k$ neighbors in $\code$, then $v$ gives charge $2/(5k)$ to $w$.
\end{itemize}

Now we simply verify that after applying the discharging rule, each vertex has charge at least $2/5$.

If $v\notin\code$, then $v$ receives charge $k(2/(5k))=2/5$.
If $v$ is a 1-cluster, then note that each neighbor $w$ of $v$ has at least one other neighbor in $\code$ (otherwise $v$ and $w$ are indistinguishable).  So $v$ gives each neighbor charge at most $1/5$.  Thus, $v$ retains charge at least $1 - 3(1/5) = 2/5$.
It is easy to see that $\code$ cannot contain 2-clusters.
Let $C$ be a $3^+$-cluster, and let $v$ be a vertex in $C$.  If $d_C(v) = 1$, then $v$ has two neighbors $v_1$ and $v_2$ not in $\code$.  Since $v_1$ and $v_2$ are distinguishable, at least one of them has another neighbor in $\code$.  
So the total charge $v$ gives to $v_1$ and $v_2$ is at most $2/5+1/5$;
thus $v$ retains charge at least $1 - 3/5 = 2/5$.  If $d_C(v) = 2$, then $v$ gives away charge at most $2/5$, so $v$ retains charge at least $1-2/5 = 3/5$.  Finally, if $d_C(v)=3$, then $v$ gives away no charge, so $v$ retains charge 1.  

Since the charge at each vertex (whether in $\code$ or not) is at least $2/5$, the density of $\code$ is at least $2/5$.
\end{proof}

It is instructive to note that our proof does not rely on the structure of the grid, but only that it is 3-regular.  In fact, Proposition~\ref{prop1} is a special case of a more general lower bound, which follows from a similar proof: any 1-ID code for any $k$-regular graph has density at least $1/(1+k/2)$.

When we study the proof of Proposition~\ref{prop1}, it is natural to look for ``slack'', i.e., vertices that have charge greater than $2/5$ after the discharging phase.  Of course every vertex $v$ in a $3^+$-cluster $C$ with at least two neighbors in $C$ has slack.  It is clear that every $3^+$-cluster $C$ contains at least one such slack vertex.  Our plan is to distribute this excess charge at the slack vertices among the vertices close to $v$.  We must also verify that each 1-cluster and each vertex not in $\code$ are close to some $3^+$-cluster.  This approach forms the outline of the proof of Theorem~\ref{mainthm}.

We think our proof of Theorem~\ref{mainthm} could be refined to give a better lower bound.  However, we think that bound would be only slightly better than Theorem~\ref{mainthm}, and that the proof would be much more complicated, so we have not attempted it.

Now we formally define density.  We fix an arbitrary vertex $v\in V(G)$ and let $V_h$ denote all the vertices that are distance at most $h$ from $v$.  The {\it density} of a code $\code$ is defined as $$\limsup_{h\to\infty}\frac{|\code\cap V_h|}{|V_h|}.$$
We note that since each vertex in $V_h$ finishes with charge at least $12/29$, the sum of the charges at vertices in $V_h$ is at least $12|V_h|/29$.  We should remark that some vertices in $V_h$ may receive charge from vertice in $\code\setminus V_h$.  However, the sum of the charges sent from vertices in $\code\setminus V_h$ to vertices in $V_h$ is linear in $h$, whereas $|V_h|$ is quadratic in $h$.  Thus, we see that $12/29$ is a lower bound on the density of $\code$.

The organization of the rest of the paper is as follows.  In Section~2 we prove the main result.  This proof consist of stating six discharging rules, and verifying that after the discharging phase, each vertex has charge at least $12/29$.  Showing that each vertex finishes with sufficient charge is a lengthy task.  To simplify the analysis, we state and prove five structural lemmas and three claims about the discharging process.  The difference between our claims and our structural lemmas is that the claims are statements about our discharging rules, whereas the lemmas are statements about any 1-ID code in the infinite hexagonal grid.  So to simplify the proof of the main result, we defer the proofs of the five structural lemmas until Section~3.

\section{Main Result}
Let $v$ be a 1-cluster.  
If $v$ has a neighbor $u\notin\code$, such that all three neighbors of $u$ are in $\code$, then we say that $v$ is {\it crowded}; otherwise $v$ is {\it uncrowded}.
Let $C$ be a 3-cluster and let $w$ be the non-leaf vertex of $C$; we call $w$ the {\it center} of $C$.  If the neighbor of $w$ that is not in $C$ has no other neighbor in $\code$, then we call $C$ an {\it open} 3-cluster and we call $w$ an {\it open} center.
Otherwise we call $C$ a {\it closed} 3-cluster. 
For each leaf $v$ of $C$, there must exist a $w\in\code$ at distance two from $v$ and not in $C$ (otherwise the two neighbors of $v$ not in $C$ would be indistinguishable).  However, a leaf may have two or more such vertices at distance two.  If any vertex $v$ in $C$, leaf or center, has at least two vertices $w_1\in\code$, $w_2\in\code$ at distance two (and neither $w_1$ nor $w_2$ is in $C$), then we call $C$ {\it crowded}; otherwise $C$ is {\it uncrowded}.
The significance of a closed or crowded 3-cluster $C$ is that $C$ will have extra help when sending charge to its adjacent vertices. 

We say that an uncrowded 1-cluster $C_1$ is {\it nearby} a $3^+$-cluster $C_2$ if $C_1$ is within distance three of either a $4^+$-cluster $C_2$, a closed 3-cluster $C_2$, or an open center of a 3-cluster $C_2$. 
Similarly, we say that an uncrowded open 3-cluster $C_1$ is {\it nearby} a $3^+$-cluster $C_2$ if $C_1$ is within distance three of either a $4^+$-cluster $C_2$ or a closed 3-cluster $C_2$, or if both its leaves are within distance three of an open 3-cluster $C_2$. 
We will show in Lemma~\ref{lem1} that each uncrowded 1-cluster $v$ is nearby a $3^+$-cluster.
%
If an open 3-cluster $C$ is uncrowded, has no open 3-clusters within distance two, and has no closed 3-clusters or $4^+$-clusters within distance three, then we say that $C$ is {\it threatened}.  
If an uncrowded 1-cluster has no $4^+$-cluster within distance three and has no nearby unthreatened 3-cluster, then we say that $v$ is {\it threatened}.
If a threatened 3-cluster $C$ has at least four nearby threatened 1-clusters and threatened 3-clusters, then we say that $C$ is {\it needy}.  Whereas clusters that are closed or crowded already have extra help sending charge, clusters that are uncrowded, threatened, or needy will likely need to receive extra charge from elsewhere.
\begin{theorem}
\label{mainthm}
The density of every vertex identifying code for the infinite hexagonal grid is at least 12/29.
\end{theorem}

\begin{proof}
Our proof is by discharging.
We assign to each vertex $v$ in the identifying code $\code$ a charge of 1.  We will redistribute the charges, without introducing any new charge, so that every vertex (whether in $\code$ or not) has charge at least 12/29.

The outline of the proof is as follows.  Consider a vertex $v$ not in $\code$.  Let $k$ be the number of neighbors of $v$ that are in $\code$.  Vertex $v$ will receive charge $12/(29k)$ from each of its $k$ neighbors (this is rule 1, in the discharging rules below).  Thus every vertex not in $\code$ receives charge $12/29$.  Clearly every neighbor of a 1-cluster $v$ must have at least two neighbors in $\code$; thus, each neighbor of $v$ will receive charge at most $6/29$ from $v$.  
If $v$ is uncrowded, then $v$ sends charge $6/29$ to each of its three neighbors, so $v$ will be left with charge $1-3(6/29) = 11/29$. Hence $v$ needs more charge. 
Our plan is to send charge 1/29 to each uncrowded 1-cluster $v$ from a nearby $3^+$-cluster $C$; we will do this via rules 2--4 below.
We will also need to prove that such a $3^+$-cluster $C$ does not send charge to too many uncrowded 1-clusters.

In verifying that each vertex finishes with charge at least $12/29$ it is convenient to count the charges of vertices in a single cluster together; i.e., for each cluster with $m$ vertices, we simply verify that the sum of the final charges of the vertices in that cluster is at least $12m/29$.
\bigskip

Note that the charge that a closed 3-cluster $C$ gives away by rule 1 is at most $3(6/29) + 2(12/29) = 42/29$.  Since $C$ begins with charge $87/29$ and needs to keep charge $36/29$, $C$ can afford to give away another $87/29 - 42/29-36/29 = 9/29$ to nearby clusters.  In contrast, the charge that an open 3-cluster $C'$ gives away by rule 1 may perhaps be as much as $2(6/29) + 3(12/29) = 48/29$.  Thus, $C'$ can only afford to give away another $3/29$ to nearby clusters.  Below, we define rules 2--4 so that each uncrowded 1-cluster receives charge $1/29$ from some nearby $3^+$-clusters.  Although rules 2--4 ensure sufficient charge for each uncrowded 1-cluster, in some cases they unfortunately require open 3-clusters to give away a total charge of $4/29$ to nearby 1-clusters.  Giving away this additional charge may result in a needy 3-cluster $C'$ with remaining charge only $35/29$ (rather than the necessary $36/29$). Thus, we add rule 5, which supplies these needy 3-clusters with the necessary charge (from nearby open $3$-clusters).
\bigskip

{\bf Discharging Rules}

\begin{enumerate}
\item Each vertex $v\notin \code$ that has $k$ neighbors in $\code$ receives charge $12/(29k)$ from each neighbor in $\code$.

\item Each uncrowded 1-cluster that is nearby a $4^+$-cluster $C$ receives charge $1/29$ from $C$.

\item Each uncrowded 1-cluster $v$ that is nearby a closed $3$-cluster $C$, and has not received charge by rule 2, receives charge $1/29$ from $C$.  

\item Each uncrowded 1-cluster $v$ that is nearby an open center in an open $3$-cluster $C$, and has not received charge by rule 2 or 3, receives charge $1/29$ from $C$.  
However, if $v$ is nearby a crowded 3-cluster $C$, then $v$ receives charge 1/29 from $C$ and receives no charge from any other cluster.
Similarly, if $v$ is not nearby a crowded 3-cluster, but lies on a 6-cycle with an open center in cluster $C$, then $v$ receives charge 1/29 from $C$ and receives no charge from any other cluster.


\item If $C_1$ is a needy, open 3-cluster 
then it has both its leaves within distance three of another open 3-cluster $C_2$; 
$C_1$ receives charge 1/29 from $C_2$.  
However, if $C_1$ and $C_2$ are both uncrowded and they each have both leaves within distance three of the other cluster, then neither cluster sends charge to the other.  We say that such $C_1$ and $C_2$ are {\it paired with each other}.
\end{enumerate}


Note that in each rule, if cluster $C_1$ receives charge from cluster $C_2$, then $C_1$ is nearby $C_2$.  This is a necessary, though not sufficient, condition for receiving charge.

Before we verify that each vertex has final charge at least $12/29$, we state five structural lemmas about the relationships between uncrowded, threatened, and needy clusters and their nearby $3^+$-clusters.  We defer the proofs of these lemmas until after we complete the discharging argument.
We separate these lemmas from the rest of the present proof because they make no mention of discharging rules.  Thus, they may be useful in proving a stronger lower bound.

\begin{lemma}
\label{lem1}
Every uncrowded 1-cluster is nearby a $3^+$-cluster.
\end{lemma}

\begin{lemma}
\label{lem2}
Every closed 3-cluster $C$ is nearby at most ten 1-clusters and open 3-clusters. If $C$ is nearby exactly ten such clusters, then $C$ is crowded.
\end{lemma}

\begin{lemma}
\label{lem5}
Every needy 3-cluster has both leaves within distance three of a $3^+$-cluster.
\end{lemma}

\begin{lemma}
Let $C_1$ and $C_2$ be uncrowded, open 3-clusters that are paired with each other.
Clusters $C_1$ and $C_2$ have at most 7 nearby threatened 1-clusters and threatened 3-clusters. If they have exactly 7 such nearby clusters, then they also have a nearby closed 3-cluster or $4^+$-cluster.
\label{lem6}
\end{lemma}

\begin{lemma}
\label{lem3}
Every cluster $C$ with $m$ vertices has at most $m+8$ nearby clusters.
\end{lemma}


Now we verify that after the discharging phase, every vertex (whether in $\code$ or not) has charge at least $12/29$; this proves that $\code$ has density at least $12/29$.  We write $f(v)$ or $f(C)$ to denote the charge at $v$ or $C$ after the discharging phase.
\bigskip

Suppose $v\notin\code$.  By rule 1, $f(v) = k(12/(29k)) = 12/29$, where $k$ is the number of neighbors of $v$ in $\code$.
Now let $v$ be a crowded 1-cluster.  One of the neighbors $u$ of $v$ has three neighbors in $\code$, so $u$ receives only charge $4/29$ from $v$.  Hence $f(v)\ge 1 - 2(6/29) - (4/29) = 13/29$.
Finally, let $v$ be an uncrowded 1-cluster.  By Lemma~\ref{lem1} and rules 2--4, $v$ receives charge $1/29$ from some nearby $3^+$-cluster.  Thus, $f(v)\ge 1 - 3(6/29) + 1/29 = 12/29$.
\bigskip

Let $C$ be a closed 3-cluster; let $u$ and $v$ be the leaves of $C$, and let $w$ be the center.  
Since we can distinguish between the neighbors of $u$ (not in $C$), at least one of them has a neighbor in $\code$ other than $u$; similarly for the neighbors of $v$.
Since $C$ is closed, the neighbor of $w$ not in $C$ has another neighbor in $\code$.  Thus, the charge given from $C$ to adjacent vertices is at most $2(12/29) + 3(6/29) = 42/29$.  
If $C$ is uncrowded, then, by Lemma~\ref{lem2}, $C$ gives charge $1/29$ to at most nine nearby clusters; thus $f(C) \ge 3 - 42/29 - 9(1/29) = 36/29$.  If $C$ is crowded, then the charge $C$ gives to adjacent vertices is at most $\max(2(6/29)+2(12/29)+4/29,4(6/29)+12/29)=40/29$.  By Lemma~\ref{lem2}, $C$ gives charge to at most ten nearby clusters, so $f(C)\ge 3 - 40/29 - 10(1/29) > 3(12/29)$.
\bigskip

Now we consider open 3-clusters.
In the remainder of this paper, we often seek to show that an open 3-cluster is not needy, thus we now study how much charge is given away by an open 3-cluster.  
\begin{claim}
Every open 3-cluster gives away charge at most 52/29.
\label{cla1}
\end{claim}

\begin{claim}
If an open 3-cluster $C$ has a vertex $v$ at distance two and $v$ does not receive charge from $C$, then $C$ gives away charge at most 51/29.
Similarly, if an open 3-cluster $C$ has a nearby closed 3-cluster or $4^+$-cluster, then $C$ gives away charge at most $51/29$ and $C$ is not needy.
\label{cla2}
\end{claim}


Since Claims~\ref{cla1} and \ref{cla2} are very similar, we only provide a proof for Claim~\ref{cla1}.  However, a short analysis of this proof yields a proof of Claim~\ref{cla2}.

\begin{figure}[htb]
\begin{center}
\begin{tikzpicture}[scale=.60]
\useasboundingbox (0,2.25) rectangle (11.0,10.1);
\grid{6}{8.7}{7};
\begin{scope}[xshift = \root cm, yshift = 3 cm, x=.866 cm, line width=.1cm, black]
\draw (3,2.5) \dot -- (4,2) \dot -- (5,2.5) \dot
(0,2) -- (0,1) -- (1,.5) (1,2.5) \dot (2,1) \dot (2,4) \dot (4,4) \dot (6,4) \dot (7,2.5) \dot (6,1) \dot (7,.5) -- (8,1) -- (8,2) 
(3,6.5) node[above] {1}
(5,6.5) node[above] {2}
(2,5) node[above] {3}
(3,5.5) node[below] {4}
(4,5) node[above] {5}
(5,5.5) node[below] {6}
(6,5) node[above] {7}
(2,4) node[below] {8}
(3,3.5) node[above] {9}
(4,4) node[below] {10}
(5,3.5) node[above] {11}
(6,4) node[below] {12}
(0,2) node[above] {13}
(1,2.5) node[below] {14}
(2,2) node[above] {15}
(3,2.5) node[below] {16}
(4,2) node[above] {17}
(5,2.5) node[below] {18}
(6,2) node[above] {19}
(7,2.5) node[below] {20}
(8,2) node[above] {21}
(0,1) node[below] {22}
(1,0.5) node[above] {23}
(2,1) node[below] {24}
(3,0.5) node[above] {25}
(4,1) node[below] {26}
(5,0.5) node[above] {27}
(6,1) node[below] {28}
(7,0.5) node[above] {29}
(8,1) node[below] {30}
(3,-.5) node[below] {31}
(4,-1) node[above] {32}
(5,-.5) node[below] {33};
\draw (3,-.5) \dot -- (4,-1) -- (5,-.5) \dot;
\draw (4.4,1.7) node[]{$C$};
\draw (4,-1.9) node[]{Fig. 1: Proof of Claim~\ref{cla1}.};
\end{scope}

\end{tikzpicture}
\end{center}
\end{figure}

\begin{proof}[of Claim~\ref{cla1}]
Let $C$ be the open 3-cluster 16-17-18 shown in Fig.~1.  
Let $U=\{8,9,10,14,15,24\}$; these are the vertices that may receive charge from $C$ and also are nearest to 16.  Similarly, let $V=\{10,11,12,19,20,28\}$ and let $W=\{26,31,33\}$.
Let $U'=U\cup\{$13-22-23$\}$, $V'=V\cup\{$21-30-29$\}$, and $W'=W\cup\{$31-32-33$\}$.

We now show that the charge $C$ sends to vertices in $U'$ is at most $20/29$ if $C$ sends charge to the 3-cluster 13-22-23\ in $U'$ and at most $19/29$ otherwise; and the same is true for $V'$ and the 3-cluster 21-30-29.  We also show that the charge $C$ sends to vertices in $W'$ is $14/29$ if $C$ sends charge to both 31 and 33, and is at most $13/29$ otherwise.

\def\u2{U\cap{\code}}
Since we can distinguish between 9 and 15, we know that $|\u2|\ge 1$.  If $|\u2|=1$, then the charge $C$ sends to $U'$ is $12/29 + 6/29+1/29 +1/29 = 20/29$ if $C$ sends charge to the 3-cluster, and $19/29$ otherwise.  If $|\u2|=2$ then either each of 9 and 15 have two neighbors in $\code$ or one of them has three neighbors and the other has only 16.  Suppose that 9 has three neighbors in $\code$; note that since 9 receives charge $4/29$ from each of its neighbors, none of these neighbors can be a threatened 1-cluster.  Thus the charge $C$ sends to $U'$ is at most $\max(2(6/29) + 2(1/29), 12/29+4/29 + 1/29)= 17/29$.  If $|\u2|=3$, then the charge $C$ sends to $U'$ is at most $6/29+4/29+1(1/29)=11/29$.   If $|\u2|=4$, then the charge $C$ sends to $U'$ is $2(4/29)+0(1/29)=8/29$.  The same analysis holds for the charge $C$ sends to $V'$.  If $|W\cap {\code}|=2$, then the charge $C$ sends to $W'$ is at most $12/29+2(1/29) = 14/29$. 
If either of 31 and 33 is not a threatened 1-cluster, then the charge $C$ sends is at most $13/29$.
Now we examine the total charge given away by $C$.

First we consider the case when both 31 and 33 are threatened 1-clusters.  Now $24\in\code$ and $28\in\code$.  Thus, $C$ does not send charge to the 3-cluster in either $U'$ or $V'$.  Hence, if $C$ is unpaired, then $C$ sends total charge at most $14/29+2(19/29)=52/29$.

Now we consider the case when at least one of 31 and 33 is not a threatened 1-cluster.  Now $C$ sends total charge at most $13/29 + 2(20/29) = 53/29$.  However, equality holds only if $10\notin\code$ and $8\in\code$ and $12\in\code$; suppose equality holds. We consider $N[10]$.  Since $9\notin\code$, $10\notin\code$, and $11\notin\code$, we must have $5\in\code$, and specifically, 5 must be in a $3^+$-cluster $C_1$.  If $C_1$ is a $4^+$-cluster, then $C_1$ gives charge to both 8 and 12, so $C$ doesn't have to.  Similary, if $C_1$ is 1-4-5 or 5-6-2, then $C_1$ is a closed 3-cluster; again $C_1$ gives charge to 8 and 12, so $C$ doesn't have to.  Finally, suppose that $C_1$ is 4-5-6.  Now $C_1$ is also an open 3-cluster.  
By the final sentence of rule 4, vertices 8 and 12 each receive charge 1/29 from $C_1$ and each receive no charge from $C$.
Thus, again $C$ gives away total charge at most 52/29.  
\end{proof}

We just proved that every open 3-cluster $C$ gives away charge at most 52/29.  Now Lemma~\ref{lem5} states that every needy 3-cluster has both leaves within distance three of a $3^+$-cluster.  So by rule 5, every unpaired needy 3-cluster receives charge $1/29$.  Thus, for every unpaired needy 3-cluster $C$,
we have $f(C)\ge 3 - 52/29 + 1/29 = 36/29$.  If an unpaired open 3-cluster is not needy, then $f(C)\ge 3 - 51/29 = 36/29$.
Hence, we now turn our attention to paired open 3-clusters.
\bigskip

Lemma~\ref{lem6} reads: ``Let $C_1$ and $C_2$ be uncrowded, open 3-clusters that are paired with each other.  Either $C_1$ and $C_2$ have at most 6 nearby threatened 1-clusters and threatened 3-clusters, or they have exactly 7 such nearby clusters, but they also have a nearby closed 3-cluster or $4^+$-cluster.''
So let $C_1$ and $C_2$ be open, uncrowded 3-clusters that are paired with each other.  We consider the two cases listed in Lemma~\ref{lem6}.  

Note that the charge that each gives to adjacent vertices not in $\code$ is $3(12/29)+2(6/29) = 48/29$.  Thus, if $C_1$ and $C_2$ have at most 6 nearby threatened 1-clusters and threatened 3-clusters, then $f(C_1)+f(C_2)\ge 6 - 2(48/29) - 6(1/29) = 72/29$.  Similarly, if $C_1$ and $C_2$ have exactly 7 nearby threatened 1-clusters and threatened 3-clusters, then $f(C_1)+f(C_2)\ge  6 - 2(48/29) - 7(1/29) + 1/29 = 72/29$. 
\bigskip

Finally, we consider $4^+$-clusters.  Let $C$ be a $4^+$-cluster with $m$ vertices.
We will show that $f(C)\ge 12m/29$.  
Note that for each $v\in C$ if $d_C(v)=1$, then $v$ gives charge at most $12/29+6/29=18/29$ to adjacent vertices not in $C$.  Similarly, if $d_C(v)=2$, then $v$ gives charge at most $12/29$ to its adjacent vertex not in $C$; if $d_C(v)=3$, then $v$ has no adjacent vertices not in $C$.  Let $\alpha_i$ denote the number of vertices $v$ in $C$ with $d_C(v)=i$ for $i=1,2,3$.   Lemma~\ref{lem3} implies that $C$ gives charge at most $(m+8)/29$ to nearby clusters.  Thus, $f(C)\ge \alpha_1 + \alpha_2 + \alpha_3 - \frac{18}{29}\alpha_1 - \frac{12}{29}\alpha_2 - \frac{1}{29}(\alpha_1+\alpha_2+\alpha_3+8)$.   We want $f(C)-12m/29\ge 0$; thus, we want to show that $(-2\alpha_1 + 4\alpha_2 + 16\alpha_3-8)/29\ge 0$.
 Note that if $\alpha_3>0$, then $\alpha_1\le \alpha_3+2$, and if $\alpha_3=0$, then $\alpha_1\le 2$. So the desired inequality holds except when $\alpha_3=0$, $\alpha_1=2$, and $\alpha_2=2$.
Now we consider the two 4-clusters with $\alpha_3=0$, $\alpha_1=2$, and $\alpha_2=2$.
%

\begin{figure}[htb]
\begin{center}
\begin{tikzpicture} [scale = .60]
\useasboundingbox (9,0.8) rectangle (12.5,10);
\begin{scope}[xshift = 0 cm]
\grid{6}{7}{6.5};
\begin{scope}[xshift = \root cm, yshift = 3 cm, x=.866 cm]
\draw[fill=black] 
(5,6.5) node[above]{1} 
(2,5) node[above]{2}
(3,5.5) node[below]{3}
(4,5) node[above]{4}
(5,5.5) node[below]{5}
(0,4) node[below]{6}
(1,3.5) node[above]{7}
(2,4) node[below]{8}
(3,3.5) node[above]{9}
(4,4) node[below]{10}
(5,3.5) node[above]{11}
(6,4) node[below]{12}
(7,3.5) node[above]{13}
(-1,2.5) node[below]{14}
(0,2) node[above]{15}
(1,2.5) node[below]{16}
(2,2) node[above]{17}
(3,2.5) node[below]{18}
(4,2) node[above]{19}
(5,2.5) node[below]{20}
(6,2) node[above]{21}
(7,2.5) node[below]{22}
(0,1) node[below]{23} 
(1,0.5) node[above]{24}
(2,1) node[below]{25}
(3,0.5) node[above]{26}
(4,1) node[below]{27}
(5,0.5) node[above]{28}
(6,1) node[below]{29}
(7,0.5) node[above]{30}
(1,-.5) node[below]{31}
(2,-1) node[above]{32}
(3,-.5) node[below]{33}
(4,-1) node[above]{34}
(5,-.5) node[below]{35}
(6,-1) node[above]{36};
\filldraw[line width=.1cm, black] (2,2)\dot -- (3,2.5) \dot -- (4,2) \dot -- (4,1) \dot 
(0,1) -- (1,.5)
(1,-.5) -- (2,-1) 
(3,-.5) -- (4,-1)
(0,2)  -- (-1, 2.5)
(0,4) -- (1,3.5)
(2,4) -- (2,5) 
(4,4) -- (4,5)
(5,3.5) -- (6,4) 
(7,2.5) -- (6,2)
(7,0.5) -- (6,1)
(6,-1) -- (5,-.5);
\draw (3.8,-2.5) node[text width=6cm]{Fig. 2a: First 4-cluster with $\alpha_3=0$, $\alpha_1=2$, and $\alpha_2=2$.};
\end{scope}
\end{scope}
\begin{scope}[xshift = 11.0 cm]
\grid{6}{8.7}{7};
\begin{scope}[xshift = \root cm, yshift = 3 cm, x=.866 cm]
\draw[fill=black] 
(2,5) node[above]{1}
(3,5.5) node[below]{2}
(4,5) node[above]{3}
(5,5.5) node[below]{4}
(6,5) node[above]{5}
(0,4) node[below]{6}
(1,3.5) node[above]{7}
(2,4) node[below]{8}
(3,3.5) node[above]{9}
(4,4) node[below]{10}
(5,3.5) node[above]{11}
(6,4) node[below]{12}
(7,3.5) node[above]{13}
(-1,2.5) node[below]{14}
(0,2) node[above]{15}
(1,2.5) node[below]{16}
(2,2) node[above]{17}
(3,2.5) node[below]{18}
(4,2) node[above]{19}
(5.05,2.5) node[below]{20}
(6,2) node[above]{21}
(7,2.5) node[below]{22}
(8,2) node[above]{23}
(0,1) node[below]{24} 
(1,0.5) node[above]{25}
(2,1) node[below]{26}
(3,0.5) node[above]{27}
(4,1) node[below]{28}
(5,0.5) node[above]{29}
(6,1) node[below]{30}
(7,0.5) node[above]{31}
(1,-.5) node[below]{32}
(2,-1) node[above]{33}
(3,-.5) node[below]{34}
(4,-1) node[above]{35}
(5,-.5) node[below]{36};
\filldraw[line width=.1cm, black]
(2,2) \dot -- (3,2.5) \dot -- (4,2)  \dot -- (5,2.5) \dot
(2,5) -- (2,4) (4,5) -- (4,4) (6,5) -- (6,4) (7,3.5) \dot
(0,4) -- (1,3.5) (-1,2.5) -- (0,2) (0,1) \dot
(1,.5) -- (1,-.5) (3,.5) -- (3, -.5) (5,.5) -- (5, -.5)
(6,1) -- (7,.5) (7,2.5) -- (8,2);
\draw (4.35,1.7) node[]{$C$};
\draw (4.4,-2.5) node[text width=6cm]{Fig. 2b: Second 4-cluster with $\alpha_3=0$, $\alpha_1=2$, and $\alpha_2=2$.};
\end{scope}
\end{scope}
\end{tikzpicture}
\end{center}
\end{figure}
\end{proof}

We begin with the 4-cluster shown in Fig.~2a.  The charge that $C$ gives to its six adjacent vertices not in $\code$ is at most $4(12/29)+2(6/29)=60/29$.  Since $C$ begins with charge 4 and must retain charge at least $4(12/29)$, the charge that $C$ can afford to give to nearby needy clusters is $4 - 60/29 - 4(12/29) = 8/29$.  Note that $C$ has at most 11 nearby needy clusters, since all the vertices at distance 2 and 3 are covered by the 11 sets $\{2,8\}$, $\{4,10\}$, $\{6,7\}$, $\{11,12\}$, $\{14,15\}$, $\{21,22\}$, $\{23,24\}$, $\{29,30\}$, $\{31,32\}$, $\{33,34\}$, and $\{35,36\}$.

If $8\in\code$, $10\in\code$, $11\in\code$, or $21\in\code$, then the charge $C$ gives to adjacent vertices is at most $54/29$, and $C$ can give charge $1/29$ to each of the at most 11 nearby needy clusters.  Hence, we assume that $8\notin\code$, $10\notin\code$, $11\notin\code$, and $21\notin\code$.  Under this assumption, we will show that $C$ has at most 8 nearby needy clusters.  Note, as follows, that the sets $\{2,8\}$ and $\{4,10\}$ intersect at most 1 needy cluster.  Since $9\notin\code$, $10\notin\code$, and $11\notin\code$, we must have 4\ in a $3^+$-cluster.  Since $8\notin\code$, we only need consider $2\in\code$.  Now either 2 and 4 are in the same cluster, or 4 is in a $4^+$-cluster or closed 3-cluster (since $12\in\code$).  By symmetry, the sets $\{11,12\}$ and $\{21,22\}$ intersect at most 1 needy cluster.
Thus, we see that $C$ has at most 9 nearby needy clusters.  
We now show that, in fact, $C$ has at most 8 nearby needy clusters.  

Suppose to the contrary that each of the seven sets $\{6,7\}$, $\{14,15\}$, $\{23,24\}$, $\{29,30\}$, $\{31,32\}$, $\{33,34\}$, and $\{35,36\}$ intersects a needy cluster.  Recall from Claim~2 that if a $3^+$-cluster $C_2$ has a vertex $v$ at distance two and $v$ is within distance three of a $4^+$-cluster, then $C_2$ is not needy.  Hence, we conclude that the cluster that intersects $\{31,32\}$ is a 1-cluster; by symmetry, we assume that this cluster is 31.  For the same reason, we must now have the clusters 23, 14, and 7.  Since 7 is a 1-cluster, we must also have $2\in\code$. Now if 2 and 4 lie in the same cluster, then that cluster is not needy, so the lemma is true.  Similarly, the lemma is true if 4 and 12 lie in the same cluster.  Hence, 4 must lie in a cluster with 5, but not with 12.  However, as before, the cluster containing 4 and 5 cannot be needy, since $12\in\code$.  Thus, the lemma is true for the first 4-cluster with $\alpha_3=0$, $\alpha_1=2$, and $\alpha_2=2$.

We now consider the 4-cluster $C$ shown in Fig.~2b.  Note that $C$ has at most 12 nearby clusters, since each of the 12 sets $\{1,8\}$, $\{3,10\}$, $\{5,12\}$, $\{6,7\}$, $\{13\}$, $\{14,15\}$, $\{22,23\}$, $\{24\}$, $\{25,32\}$, $\{27,34\}$, $\{29,36\}$ and $\{30,31\}$ intersects at most one cluster, and these 12 sets cover all the vertices at distance two or three from $C$.  As for the previous case, if any of $8\in\code$, $10\in\code$, $27\in\code$, or $29\in\code$ hold, then $C$ gives away charge at most $54/29$ to adjacent vertices, so $C$ can afford to give away charge $1/29$ to each of the at most 12 nearby clusters.  Thus, we assume that $8\notin\code$, $10\notin\code$, $27\notin\code$, and $29\notin\code$.

We will show that the four sets $\{1,8\}$, $\{3,10\}$, $\{5,12\}$, and $\{13\}$ intersect at most two needy clusters.  Since $N[10]\cap\code\ne\emptyset$, we know that $3$ is in a $3^+$-cluster $C_1$; we consider two (non-exclusive) cases: $2\in C_1$ and $4\in C_1$.  First suppose that $2\in C_1$. Now the two sets $\{1,8\}$ and $\{3,10\}$ intersect at most one cluster.  Further, if $5\in\code$, then we can show that $C_1$ is not needy, since 5 is nearby $C_1$ but the cluster contain 5 receives no charge from $C_1$, since it is also nearby $C$.  However, if $5\notin\code$, then the four sets $\{1,8\}$, $\{3,10\}$, $\{5,12\}$, and $\{13\}$ intersect at most two clusters. 
Suppose instead that $4\in C_1$.  If $1\in\code$, then we can show that $C_1$ is not needy.  However, if $1\notin\code$, then the four sets $\{1,8\}$, $\{3,10\}$, $\{5,12\}$, and $\{13\}$ intersect at most two clusters. 

Thus, the four sets $\{1,8\}$, $\{3,10\}$, $\{5,12\}$, and $\{13\}$ intersect at most two needy clusters.  By symmetry,  the four sets $\{24\}$, $\{25,32\}$, $\{27,34\}$, and $\{29,36\}$ intersect at most two needy clusters.   Thus, $C$ has at most 8 nearby, needy clusters, so $f(C)\ge 4 - 60/29 - 8(1/29) = 48/29$.
This concludes the proof of Theorem~\ref{mainthm}, subject to proving Lemmas~\ref{lem1}--\ref{lem3}, which we do in the next section.  

\section{Structural Lemmas}
\setcounter{lemma}{0}

\begin{figure}[htb]
\begin{center}
\begin{tikzpicture} [scale = .60]
\useasboundingbox (9,2.2) rectangle (10,10);
\grid{6}{7.0}{5.2};
\begin{scope}[xshift = \root cm, yshift = 3 cm,x=.866 cm]
\draw[fill=black,line width=.1cm]
(2,5) node[above]{1} 
(4,5) node[above]{2} 
(2,4) node[below]{3} 
(3,3.5) \dot node[above]{4} 
(4,4) node[below]{5} 
(1,2.5) \dot node[below]{6} 
(2,2) node[above]{7} 
(3,2.5) node[below]{8} 
(4,2) node[above]{9} 
(5,2.5) \dot node[below]{10}
(2,1) \dot node[below]{11} 
(3,0.5) node[above]{12}
(4,1) \dot node[below]{13};
\draw (3,-1.9) node[]{Fig. 3a: Proof of Lemma~\ref{lem1}.};
\end{scope}
\begin{scope}[xshift = 9.5 cm]
\grid{6}{8.7}{7};
\begin{scope}[xshift = \root cm, yshift = 3 cm, x=.866 cm]
\draw[fill=black] 
(5,6.5) node[above]{1} 
(2,5) node[above]{2}
(3,5.5) node[below]{3}
(4,5) node[above]{4}
(5,5.5) node[below]{5}
(6,5) node[above]{6}
(7,5.5) node[below]{7}
(1,3.5) node[above]{8}
(2,4) node[below]{9}
(3,3.5) node[above]{10}
(4,4) node[below]{11}
(5,3.5) node[above]{12}
(6,4) node[below]{13}
(7,3.5) node[above]{14}
(0,2) node[above]{15}
(1,2.5) node[below]{16}
(2,2) node[above]{17}
(3,2.5) node[below]{18}
(4,2) node[above]{19}
(5,2.5) node[below]{20}
(6,2) node[above]{21}
(7,2.5) node[below]{22}
(8,2) node[above]{23}
(1,0.5) node[above]{24}
(2,1) node[below]{25}
(3,0.5) node[above]{26}
(4,1) node[below]{27}
(5,0.5) node[above]{28}
(6,1) node[below]{29}
(7,0.5) node[above]{30}
(3,-.5) node[below]{31}
(5,-.5) node[below]{32};
\filldraw[line width=.1cm, black] (3,2.5) \dot -- (4,2) \dot -- (5,2.5) \dot
(3,.5) -- (3,-.5)
(1,.5) -- (2, 1)
(0,2)  -- (1, 2.5)
(1,3.5) \dot 
(2,4) -- (2,5) 
(4,4) -- (4,5)
(6,5) -- (6,4) 
(7,3.5) \dot
(7,2.5) -- (8,2)
(7,0.5) -- (6,1)
(5,0.5) -- (5,-.5);
\draw (4.35,1.7) node[]{$C$};
\draw (4,-1.9) node[]{Fig. 3b: Proof of Lemma~\ref{lem2}.};
\end{scope}
\end{scope}
\end{tikzpicture}
\end{center}
\end{figure}

\begin{lemma}
Every uncrowded 1-cluster is nearby a $3^+$-cluster.
\end{lemma}
\vspace{-.15in}
\begin{proof}
Let $13\in\code$ be the uncrowded 1-cluster shown in Fig.~3a.  Since we can distinguish $13$ from its neighbors, each of these neighbors has an additional neighbor in $\code$.  By symmetry, we may assume that $10,11\in \code$.  If 10 or 11 is in a $3^+$-cluster $C$, then the lemma holds; hence we may assume that 10 and 11 are 1-clusters.  Since 13 is uncrowded, and $10,13\in\code$, we know that $8\notin\code$.  Since we can distinguish 7 from 11, we know that $6\in\code$.  Since $N[8]\cap\code$ is nonempty, we know that $4\in\code$.  Since we can distinguish 4 from 8, we know that 4 is in a $3^+$-cluster.  If 4 is in a $4^+$-cluster or closed 3-cluster $C$, then the lemma holds.  Hence, $C=$ 3-4-5.  Now 4 is an open center, so again the lemma holds.
\end{proof}

\begin{lemma}
Every closed 3-cluster $C$ has at most ten nearby 1-clusters and open 3-clusters.  If $C$ has exactly ten such clusters, then $C$ is crowded.
\end{lemma}

\begin{proof}
Let $C$ be the 3-cluster 18-19-20 shown in Fig.~3b.  
Consider the nine bold edges and two bold isolated vertices not in $C$.  These eleven sets cover all of the vertices at distance two or three from $C$.  Hence, $C$ has at most 11 distinct nearby clusters (this is a special case of Lemma~\ref{lem3}).   Note that either 4 is in a $3^+$-cluster $C_1$ or 11 is a 1-cluster.  We first consider the possibilities for $C_1$.
By symmetry, we assume that $5\in C_1$; additionally, either $1\in C_1$, $3\in C_1$, $6\in C_1$, or $11\in C_1$.

Suppose $\{4,5,11\}\subseteq C_1$.  Note that $\{6,13,14\}$ intersects at most one cluster (other than $C_1$).  Similarly, if $\{2,8,9\}$ intersects two nearby clusters, then $C_1$ is a closed 3-cluster or $4^+$-cluster.  
Hence, $C$ has at most nine nearby 1-clusters and open 3-clusters.

Suppose $\{3,4,5\}\subseteq C_1$.  Note that $\{2,8,9\}$ intersects at most one cluster other than $C_1$; similarly for $\{6,13,14\}$.
Hence, $C$ has at most nine nearby 1-clusters and open 3-clusters.

Suppose $\{4,5,6\}\subseteq C_1$.  Assume also that $11\notin C_1$, since this is an earlier case.  Clearly, $C$ has at most ten nearby clusters.  Furthermore, equality holds only if each bold set (other than $\{4,11\}$ and $\{6,13\}$) intersects a distinct nearby cluster.  Under this assumption, we note the following: $14\in\code$, $22\notin\code$, $29\in\code$ (since each leaf of $C$ has some vertex $v$ at distance two, with $v\in\code$), $28\notin\code$, $26\in\code$ (since $C$ is closed), $25\notin\code$.  Also, $8\in\code$ implies that $9\notin\code$ and $16\notin\code$.  But now 18 has no vertex $v$ at distance two, with $v\in\code$.  Hence $C$ has at most nine nearby clusters.

Suppose $\{1,4,5\}\subseteq C_1$.  Also assume that $6\notin\code$, since otherwise we are in an earlier case.  Note that $13\notin\code$, for if it is, then 14 cannot be a distinct cluster and also $C_1$ is a closed 3-cluster or $4^+$-cluster.  Since $6\notin\code$ and $13\notin\code$, we are in a similar situation to the previous case, and that argument suffices.

Hence, we conclude that 11 is a 1-cluster.  
First we prove that the five sets $\{8\}$, $\{2,9\}$, $\{4,11\}$, $\{6,13\}$, and $\{14\}$ intersect at most four nearby 1-clusters and open 3-clusters.  Second, we assume that $C$ is not crowded and show that the other six bold sets intersect at most five nearby 1-clusters and open 3-clusters.
Assume that both $8\in\code$ and $14\in\code$.  Since 4 is distinguishable from 11, either $5\in\code$ or $3\in\code$; by symmetry, assume $5\in\code$.  If $6\notin\code$, then $\{6,13\}$ intersects no cluster other than (possibly) the one that contains 14.
However, if $6\in\code$, then the cluster that contains 5 and 6 is either a closed 3-cluster or a $4^+$-cluster.  Hence, $C$ has at most 10 nearby 1-clusters and open 3-clusters.

Now assume further that $C$ is not crowded.  Since $C$ is closed, either $26\in\code$ or $28\in\code$; by symmetry, assume $26\in\code$.  Since $C$ is not crowded, we know that 
$9\notin\code$, $13\notin\code$, $16\notin\code$, $22\notin\code$, $25\notin\code$, $28\notin\code$, and $29\notin\code$.  Since $N[29]\cap\code$ is nonempty, 30 is in a $3^+$-cluster $C_1$.  Now either $23\notin\code$, 
or $23\in C_1$, 
or $32\notin\code$, 
or $C_1$ is a closed 3-cluster (due to 32) or a $4^+$-cluster.  In every case, the lemma is true.
\end{proof}

\begin{lemma}
Every needy 3-cluster has both leaves within distance three of a $3^+$-cluster.
\end{lemma}
\begin{figure}
\begin{center}
\begin{tikzpicture}[scale=.6]
\useasboundingbox (4,2.0) rectangle (6.5,8);
\grid{6}{8.7}{7};
\begin{scope}[xshift = \root cm, yshift = 3 cm, x=.866 cm]
\draw[fill=black] 
(5,6.5) node[above]{0}
(4,5) node[above]{1}
(5,5.5) node[below]{2}
(6,5) node[above]{3}
(7,5.5) node[below]{4}
(8,5) node[above]{5}
(4,4) node[below]{6}
(5,3.5) node[above]{7}
(6,4) node[below]{8}
(7,3.5) node[above]{9}
(8,4) node[below]{10}
(9,3.5) node[above]{11}
(6,2) node[above]{12}
(7,2.5) node[below]{13}
(8,2) node[above]{14}
(3,-.5) node[below]{15}
(5,-.5) node[below]{16};
\draw[line width=.1cm, black] (3,2.5) \dot -- (4,2) \dot -- (5,2.5) \dot;
\draw[line width=.1cm, black] (3,-.5) \dot -- (4,-1) \dot -- (5,-.5) \dot;
\draw[line width=.1cm, black] (0,2) \dot -- (0,1) \dot -- (1,.5) \dot;
\draw[line width=.1cm, black] (8,2) \dot -- (8,1) \dot -- (7,.5) \dot;
\draw (4.55,1.7) node[]{$C_1$};
\draw (8,.95) node[below]{$C_2$};
\draw (4,-1.9) node[]{Fig. 4: Proof of Lemma~\ref{lem5}.};
\end{scope}
\end{tikzpicture}
\end{center}
\end{figure}
\begin{proof}
Let $C_1$ be the 3-cluster shown in Fig.~4.  Since $N[6]\cap\code$ is nonempty, either 
1 is in a $3^+$-cluster, or 6 is a 1-cluster.  
In the former case, the lemma holds, so we consider only the latter case.

Suppose 6 is a 1-cluster.  
We will now show that $C_1$ is not needy; suppose, for contradiction, that $C_1$ is needy.  By definition, $C_1$ is uncrowded.  This implies that neither 15 nor 16 is a 1-cluster.  Thus, the only candidates to be nearby threatened 1-clusters and threatened 3-clusters, other than 6, are the three 3-clusters shown in bold.  Furthermore, since $C_1$ is needy, 6 and each of these 3-clusters must be threatened.
Since 6 is a 1-cluster, 1 must also have another neighbor in $\code$; by symmetry, we assume $2\in\code$.
Now we consider whether 9 is in $\code$ or not.  

Suppose that $9\in\code$.  Now 9 must be a 1-cluster, since $C_2$ is threatened.  Since 8 is distinguishable from 9, we know that $3\in\code$.  Now we have $2\in\code$ and $3\in\code$, so 2 and 3 lie together in a $3^+$-cluster $C_3$.   Note that since $6\in\code$ and $9\in\code$, $C_3$ is either a closed 3-cluster or a $4^+$-cluster.  In each case, 6 is unthreatened, so $C_1$ is not needy.

Suppose instead that $9\notin\code$.  Since $7\notin\code$, $8\notin\code$, and $9\notin\code$, we know that $3\in\code$.  Similarly, $8\notin\code$, $9\notin\code$, and $13\notin\code$, so $10\in\code$; furthermore, 10 is in a $3^+$-cluster.  Since $C_2$ is threatened, 11 cannot be in a $3^+$-cluster with 10; hence, 10 is in a cluster with 5.  Now we consider the cluster $C_3$ that contains 2 and 3.  Again, $C_3$ is either a closed 3-cluster or a $4^+$-cluster.  In each case, 6 is unthreatened, so $C$ is not needy.
\end{proof}

\begin{lemma}
Let $C_1$ and $C_2$ be uncrowded, open 3-clusters that are paired with each other.
Clusters $C_1$ and $C_2$ have at most 7 nearby threatened 1-clusters and threatened 3-clusters. If they have exactly 7 such nearby clusters, then they also have a nearby closed 3-cluster or $4^+$-cluster.
\end{lemma}
\begin{figure}[htb]
\begin{center}
\begin{tikzpicture}[scale=.60]
\useasboundingbox (2,1.8) rectangle (10,10);
\grid{6}{10.7}{9};
\begin{scope}[xshift = \root cm, yshift = 3 cm, x=.866 cm]
\draw[fill=black] 
(3,6.5) node[above]{1}
(4,7) node[below]{2}
(5,6.5) node[above]{3}
(6,7) node[below]{4}
(7,6.5) node[above]{5}
(8,7) node[below]{6}
(0,5) node[above]{7}
(1,5.5) node[below]{8}
(2,5) node[above]{9}
(3,5.5) node[below]{10}
(4,5) node[above]{11}
(5,5.5) node[below]{12}
(6,5) node[above]{13}
(7,5.5) node[below]{14}
(8,5) node[above]{15}
(9,5.5) node[below]{16}
(10,5) node[above]{17}
(0,4) node[below]{18}
(1,3.5) node[above]{19}
(2,4) node[below]{20}
(3,3.5) node[above]{21}
(4,4) node[below]{22}
(5,3.5) node[above]{23}
(6.1,4) node[below]{24}
(7,3.5) node[above]{25}
(8,4) node[below]{26}
(9,3.5) node[above]{27}
(10,4) node[below]{28}
(11,3.5) node[above]{29}
(2,2) node[above]{30}
(3,2.5) node[below]{31}
(4,2) node[above]{32}
(5,2.5) node[below]{33}
(6,2) node[above]{34}
(7,2.5) node[below]{35}
(8,2) node[above]{36}
(9,2.5) node[below]{37}
(10,2) node[above]{38}
(4,1) node[below]{39}
(5,0.5) node[above]{40}
(6.1,1) node[below]{41}
(7,0.5) node[above]{42}
(8,1) node[below]{43}
(9,0.5) node[above]{44}
(3,-.5) node[below]{45}
(7,-.5) node[below]{46};
\draw[line width=.1cm, black] (3,3.5) \dot -- (4,4) \dot -- (5,3.5) \dot;
\draw[line width=.1cm, black] (4,1) \dot -- (5,.5) \dot -- (6,1) \dot;
\draw (5.0,-1.9) node[]{Fig. 5: Proof of Lemma~\ref{lem6}.};
\draw (5.5,3.1) node[]{$C_1$};
\draw (5.5,0.1) node[]{$C_2$};
\end{scope}
\end{tikzpicture}
\end{center}
\end{figure}
\begin{proof}
Let $C_1 =$ 21-22-23 and $C_2=$ 39-40-41 be the paired, uncrowded, open 3-clusters shown in Fig.~5.  
We note that the number of 1-clusters and open 3-clusters nearby $C_1$ and $C_2$ is at most 8, as follows.  
We show that $C_1$ has at most four such nearby clusters.
First suppose that at most one of 1 and 3 is a 1-cluster.  
Now at most one nearby cluster intersects each of the following four vertex sets: $\{1,2,3\}$, $\{7,8,9,18,19\}$, $\{13,14,15,25,26\}$, and $\{30\}$.  If both 1 and 3 are 1-clusters, then the same analysis holds, except now we know that $9\in\code$, so $30\notin\code$.

Since we can distinguish 24 from 33, we know that either $13\in\code$ or $25\in\code$.  Similarly, since we can distinguish 34 from 42, we know that $35\in\code$, $43\in\code$, or $46\in\code$.  Note that if $25\in\code$ and $35\in\code$, then 25 and 35 are in the same $3^+$-cluster and it
must be a closed 3-cluster or $4^+$-cluster. 
Furthermore, $C_1$ and $C_2$ have at most 7 other nearby clusters; so the lemma holds.
Thus, we have five possibilities: $\{13,35\}$, $\{13,43\}$, $\{13,46\}$, $\{25,43\}$, and $\{25,46\}$.  We consider each possibility in turn, but we defer the first possibility until the end.

Suppose $13\in\code$ and $43\in\code$.  Since $C_2$ is uncrowded, $35\notin\code$; similarly $25\notin\code$.  Thus $26\in\code$ and $36\in\code$.  Now 36 and 43 lie together in a $3^+$-cluster $C_3$.  If $C_3 =$ 36-43-44, then $C_3$ is a closed 3-cluster and again $C_1$ and $C_2$ have at most 7 other nearby clusters.  
The situation is similar if $C_3$ is a $4^+$-cluster.  Hence, we assume that $C_3 =$ 37-36-43.  Since we can distinguish between 25 and 26, we know that 26 is in a $3^+$-cluster $C_4$.  
Furthermore, $C_4$ is either a closed 3-cluster or a $4^+$-cluster; thus, both $C_3$ and the cluster containing 13 are unthreatened.

Suppose $13\in\code$ and $46\in\code$.  The analysis is similar to the previous case.  Since $C_1$ is not crowded, $25\notin\code$.  Similarly, $35\notin\code$.  Thus $26\in\code$ and $36\in\code$ and each are in $3^+$-clusters.  Since $C_2$ is not crowded, $43\not\in\code$; thus, 36 lies in a $3^+$-cluster $C_3$ with 37.  
Now $C_3$ is either a closed 3-cluster or a $4^+$-cluster; thus, the cluster containing 46 is not threatened.
Similarly, the cluster containing 26 is either a $4^+$-cluster or a closed 3-cluster, so the cluster containing 13 is not threatened.

Suppose $25\in\code$ (and either $43\in\code$ or $46\in\code$).  Since $C_1$ is open and uncrowded, $12\notin\code$ and $13\notin\code$.  Since $N[12]\cap\code\ne\emptyset$, we know $3\in\code$.  Also 14 is in a $3^+$-cluster $C_3$.  If $C_3$ is closed or a $4^+$-cluster, then the clusters containing 3 and 25 are unthreatened, so the lemma holds.  Thus, we assume that $C_3=$ 5-14-15.  
Since 3 is distinguishable from 12, we know that 3 is in a $3^+$-cluster $C_4$.  
Since $5\in\code$, $C_4$ is unthreatened; so we may assume that all other nearby clusters are threatened.  Thus, 25 is a 1-cluster.  So, $36\in\code$.  If $43\in\code$, then the cluster containing 36 and 43 is nearby $C_2$ and is a closed 3-cluster or $4^+$-cluster.  Hence, we assume that $46\in\code$.  Since 25 is threatened, 5-14-15 is uncrowded.  Thus, $27\notin\code$.  
So to distinguish between 36 and 37, we must have $38\in\code$.  Suppose $37\notin\code$.  Now 28 is in a $3^+$-cluster $C_5$.  Recall that $17\notin\code$, since 5-14-15 is uncrowded.  Thus, the cluster containing 28 must be either a closed 3-cluster or a $4^+$-cluster; in each case, 25 is unthreatened, so the lemma is true.
Thus $37\in\code$; we may assume 36-37-38 is an open 3-cluster, since otherwise 25 is unthreatened.  Now since $17\notin\code$ and $28\notin\code$, we have $29$ in a $3^+$-cluster.  However, now 36-37-38 has a $3^+$-cluster at distance two, so 25 is unthreatened.  Again the lemma holds.
\bigskip

\begin{figure}[htb]
\begin{center}
\begin{tikzpicture}[scale=.60]
\useasboundingbox (2,1.7) rectangle (10,13);
\grid{8}{10.7}{9};
\begin{scope}[xshift = \root cm, yshift = 6 cm, x=.866 cm]
\draw[line width=.1cm, black] (3,3.5) \dot -- (4,4) \dot -- (5,3.5) \dot
(6,5) \dot 
(0,4) \dot -- (0,5) \dot -- (1,5.5) \dot 
(3,6.5) \dot -- (4,7) \dot -- (5,6.5) \dot 
(4,1) \dot -- (5,.5) \dot -- (6,1) \dot 
(7,2.5) \dot 
(2,2) \dot 
(3,-.5) \dot 
(9,.5) \dot -- (9,-.5) \dot -- (8,-1) \dot 
(4,-2) \dot -- (5,-2.5) \dot -- (6,-2) \dot; 
\draw (5.3,-4.9) node[]{Fig. 6: Proof of Lemma~\ref{lem6} (continued).};
\draw[fill=black] 
(3,6.5) node[above]{1}
(4,7) node[below]{2}
(5,6.5) node[above]{3}
(6,7) node[below]{4}
(7,6.5) node[above]{5}
(8,7) node[below]{6}
(0,5) node[above]{7}
(1,5.5) node[below]{8}
(2,5) node[above]{9}
(3,5.5) node[below]{10}
(4,5) node[above]{11}
(5,5.5) node[below]{12}
(6,5) node[above]{13}
(7,5.5) node[below]{14}
(8,5) node[above]{15}
(9,5.5) node[below]{16}
(10,5) node[above]{17}
(0,4) node[below]{18}
(1,3.5) node[above]{19}
(2,4) node[below]{20}
(3,3.5) node[above]{21}
(4,4) node[below]{22}
(5,3.5) node[above]{23}
(6.1,4) node[below]{24}
(7,3.5) node[above]{25}
(8,4) node[below]{26}
(9,3.5) node[above]{27}
(10,4) node[below]{28}
(11,3.5) node[above]{29}
(2,2) node[above]{30}
(3,2.5) node[below]{31}
(4,2) node[above]{32}
(5,2.5) node[below]{33}
(6,2) node[above]{34}
(7,2.5) node[below]{35}
(8,2) node[above]{36}
(9,2.5) node[below]{37}
(10,2) node[above]{38}
(4,1) node[below]{39}
(5,0.5) node[above]{40}
(6.1,1) node[below]{41}
(7,0.5) node[above]{42}
(8,1) node[below]{43}
(9,0.5) node[above]{44}
(3,-.5) node[below]{45}
(7,-.5) node[below]{46}
(8,-1) node[above]{47}
(9,-.5) node[below]{48}
(4,-2) node[below]{49}
(5,-2.5) node[above]{50}
(6,-2) node[below]{51};
\draw (5.47,3.1) node[]{$C_1$};
\draw (5.47,0.1) node[]{$C_2$};
\end{scope}
\end{tikzpicture}
\end{center}
\end{figure}

Hence we may assume that $13\in\code$ and $35\in\code$; by symmetry, we may also assume that $30\in\code$ and $45\in\code$ (see Fig.~6).  Observe that we have eight candidates to be nearby threatened clusters; these are 13, 30, 35, 45, 18-7-8, 47-48-44, 3 or 3-2-1, and 49 or 49-50-51.  If at most six of these candidates are threatened, then the lemma holds.  Hence, by symmetry, we may assume that each of 13, 35, 18-7-8, and 3 or 3-2-1 is threatened.  Since 1 is distinguishable from 10, we know that 1 is in a $3^+$-cluster $C_3$.  If $2\notin C_3$, then $C_3$ is a $4^+$-cluster or a closed 3-cluster; so the cluster containing 3 is unthreatened.  Thus, we assume that $C_1$ is 1-2-3.  
Since 13 and 14 are distinguishable, either $5\in\code$ or $15\in\code$.  If $5\in\code$, then 1-2-3 is crowded, so neither 1-2-3 nor 13 is threatened; hence, we assume $15\in\code$.
Since 35 is distinguishable from 25, we know that $26\in\code$.  However, now the cluster containing 15 and 26 is either a closed 3-cluster or a $4^+$-cluster; in each case both 13 and 35 are unthreatened.
\end{proof}

\begin{lemma}
Every cluster $C$ with $m$ vertices has at most $m+8$ nearby clusters.
\end{lemma}
Instead of proving Lemma~\ref{lem3} directly, we prove Lemma~\ref{lem4}, which is stronger, but also easier to prove by induction.
\begin{lemma}
\label{lem4}
If $C$ is a cluster with $m$ vertices, then we can partition the set of vertices at distances two and three from $C$ into $m+8$ sets, each of which consists of two adjacent vertices or a single vertex.  Furthermore, if a vertex $v$ is distance three from $C$ and in fact has two disjoint paths of length 3 to the same vertex $u\in C$, then $v$ is in a set of size 1\ in the partition.
\end{lemma}
Lemma~\ref{lem4} immediately implies Lemma~\ref{lem3}, since each set in the partition intersects at most one cluster.  If a vertex $v$ is distance at most three from $C$, then either $v$ is adjacent to $C$ or $v$ is in the partition for $C$; we call such a vertex $v$ {\it covered}.  If a vertex is not covered, then we call it {\it uncovered}.
\begin{figure}[thb]
\begin{center}
\begin{tikzpicture}[scale=.60]
\useasboundingbox (9.5,2.10) rectangle (14,7);
\begin{scope}[yshift = 0 cm, xshift = 1cm]
\gridsmall{7.0}{5.2};
\begin{scope}[xshift = \root cm, yshift = 3 cm, x=.866 cm]
\draw[fill=black]
(1,2.5) node[below]{1} 
(2,2) node[above]{2} 
(3,2.5) node[below]{3} 
(4,2) node[above]{4} 
(5,2.5) node[below]{5} 
(2,1) node[below]{6} 
(3,.5) node[above]{7} 
(4,1) node[below]{8} 
(5,.5) node[above]{9} 
(6,1) node[below]{10}
(2,-1) node[above]{11} 
(3,-.5) node[below]{12};
\draw (3,-1.9) node[]{Fig. 7a: The first induction step.}; 
\draw[line width=.1cm, fill=black] 
(2,-1) \dot -- (3, -.5) \dot -- (3,.5) \dot;
\end{scope}
\end{scope}
\begin{scope}[xshift = 15cm]
\gridsmall{5.2}{3.5};
\begin{scope}[xshift = \root cm, yshift = 3 cm, x=.866 cm]
\draw[fill=black, line width=.1cm]
(1,3.5) node[above]{1} 
(3,3.5) node[above]{2} 
(1,2.5) node[below]{3} 
(2,2) node[above]{4} 
(3,2.5) node[below]{5} 
(1,.5) \dot node[above]{6} 
(2,1) node[below]{7} 
(3,.5) \dot node[above]{8};
\draw (1.5,-1.9) node[]{Fig. 7b: The second induction step.}; 
\end{scope}
\end{scope}
\end{tikzpicture}
\end{center}
\end{figure}

\begin{proof}
We use induction on the size of $C$, growing $C$ by one vertex at each step.  The base case $|C|=3$ is easy and is shown in Fig.~3b.  Let $C'$ be a cluster of size $k+1$ and let $T$ be a spanning tree of $C'$ with a leaf $v$.  By deleting $v$ we reach a cluster $C$, of size $k$, for which the induction hypothesis holds; let $\part$ be the desired partion for $C$ of size $k$.  
We consider different induction steps, depending on whether $v$ has one, two, or three neighbors in $C$.  Note that if $v$ has three neighbors in $C$, then $\part$ is still a valid partition for $C'$.  So we consider below the cases when $v$ has one or two neighbors in $C$.

First we consider a cluster $C'$ that is built from $C$ by adding a vertex with only one neighbor in $C$.  Let 7 be the new vertex (see Fig.~7a); by symmetry, we may assume that $11,12\in C$.  We assume that 1, 3, 5 are uncovered (the case where one or more of these vertices is covered is easier, so we omit the details).  We will modify $\part$ to form a new partition that also includes 1, 3, 5.  Beginning with $\part$, we delete the sets that contain 2 and 4, and we add the sets $\{1,2\}$, $\{3\}$, and $\{4,5\}$.  
As above, we assume that 10 is uncovered (since the other case is easier).
Since 10 is uncovered, $\part$ contains the set $\{9\}$; replace the set $\{9\}$ with the set $\{9,10\}$.  We now have a partition $\part'$ for $C'$ and $|\part'|=k+1$.  

Second we consider a cluster $C'$ that is built from $C$ by adding a vertex with two neighbors in $C$.  Let 7 be the new vertex added to $C$ and let 6 and 8 be vertices already in $C$ (see Fig.~7b).  If 1 is uncovered, then 3 is distance three from $C$; similarly, if 2 is uncovered, then 5 is distance three from $C$.  We assume that 1 and 2 are uncovered by $C$ (the case where one or both of 1 and 2 is covered is easier, so we omit the details).  So 3 and 5 are both in sets of size one; we thus replace $\{3\}$ with $\{1,3\}$ and we replace $\{5\}$ with $\{2,5\}$.  
We now have a partition $\part'$ for $C'$ and $|\part'|=k$.  
\end{proof}

\section{Acknowledgement}
Thank you to Ryan Martin, who brought this problem to our attention.
Thank you also to an anonymous referee who noticed inconsistencies and suggested improvements.

\end{document}